\documentclass{amsart}
\usepackage[foot]{amsaddr}
\usepackage[english]{babel}
\usepackage{amsthm}

\usepackage{mathtools}
\usepackage{thmtools}
\usepackage{thm-restate}
\usepackage{enumerate} 
\usepackage{hyperref}
\usepackage{xfrac}
\usepackage{cleveref}
\usepackage{verbatim}
\usepackage{amsmath}
\usepackage{amssymb}
\usepackage{mathabx}
\usepackage{graphicx}
\usepackage{amsfonts}
\usepackage{pb-diagram}
\usepackage{tikz-cd}
\usepackage{graphicx}
\usepackage{xcolor}
\usepackage{adjustbox}
\usepackage{mathtools}
\newtheorem{theorem}{Theorem}
\numberwithin{theorem}{section}

\newtheorem{definition}[theorem]{Definition}

\newtheorem{lemma}[theorem]{Lemma}
\newtheorem{observation}[theorem]{Observation}
\newtheorem{fact}[theorem]{Fact}

\newtheorem{question}[theorem]{Question}
\newtheorem{claim}[theorem]{Claim}

\newenvironment{subproof}[1][\proofname]{%
\begin{proof}[#1]%
	}{%
\end{proof}%
}

\newenvironment{subsubproof}[1][\proofname]{%
\begin{proof}[#1]%
}{%
\end{proof}%
}

\newcommand{\ja}[1]{{\color{magenta}#1}}
\newcommand{\p}[0]{\mathcal{P}}
\newcommand{\nin}[0]{\notin}
\newcommand{\nb}[2]{N_{#2}(#1)}
\newcommand{\N}[0]{\mathbb{N}}

\newcommand{\G}[0]{Homeo_{\partial S}(S)}

\newcommand{\MCG}[0]{PMod}

\title{Non-Roelcke precompactness of groups of surface homeomorphisms}

\author{J. de la Nuez Gonz\'alez }
\email{jnuezgonzalez@gmail.com}
\address{Korean Institute of Advanced Study (KIAS)}
\date{\today}
\thanks{Work supported by Samsung Science and Technology Foundation under Project Number SSTF-BA1301-51.  }
\begin{document}

\begin{abstract}
We prove that no subgroup of the group of boundary-fixing homeomorphisms of a compact surface whose action on the interior of the surface is sufficiently transitive can be Roelcke precompact with the topology inherited from the compact-open topology.
\end{abstract}

\maketitle

\section{Introduction}
  
  Roelcke precompactness is the topological group counterpart of the model theoretic notion of $\omega$-categoricity in classical and continuous logic. As such, it has received a lot of attention in the literature from a variety of angles.
  
  \begin{definition}
  We say that a topological group $G$ is Roelcke precompact if for every neighbourhood of the identity $U$ there exists a finite subset $g_{1},g_{2},\dots g_{k}$ such that $G=\bigcup_{j=1}^{k}Ug_{j}U$.
  \end{definition}
  

  It is a result of Culler and Rosendal \cite{rosendal2013global} that the homeomorphism group of a manifold of dimension $\geq 2$ (and more generally, any group of homeomorphisms containing a specific class of elements) is not Roelcke precompact. In this note we elaborate on the general idea of their proof to show that the same property holds for a large class of groups of homeomorphisms of compact surfaces.

  Given a surface with (possibly empty) boundary $S$, from now on simply ´a surface', and some tuple $p_{1},p_{2},\dots p_{r}$ of points in the interior of $S$, $int(S)$, we denote by $S_{p_{1},p_{2},\dots p_{r}}$ the result of removing $p_{1},\dots p_{r}$ (the punctures) from $S$. We write simply $S_{r}$, $r>0$ when only interested in the homeomorphism type.
  
  The pure mapping class group of a punctured surface $S_{p_{1},\dots p_{r}}$, or $\MCG(S_{p_{1},\dots p_{r}})$ is the quotient of the pointwise stabilizer of the set of punctures in $Homeo_{\partial S}(S)$ (whose elements we can regard alternatively as homeomorphisms of $S_{p_{1},p_{2}\dots p_{m}}$ fixing the boundary) by identifying any two elements which are homotopic via a homotopy fixing the punctures. Given $g\in S_{p_{1},\dots p_{m}}$ we will denote the corresponding mapping class by $[g]$. 
  
  Accordingly, there is a natural quotient map 
  $\MCG(S_{p_{1},\dots p_{r}})\to \MCG(S_{p_{1},\dots p_{r-1}})$.
  The kernel is given by the image in $\MCG(S_{p_{1},p_{2},\dots p_{r}})$ of the point pushing map $\mathcal{P}:\pi(S_{p_{1},p_{2},\dots p_{r-1}},p_{r})\to \MCG(S_{p_{1},p_{2},\dots p_{r}})$, which plays a certain role in this paper and about which we say a few more words at the beginning of section \ref{s: 2}.

  For simplicity we assume $S$ is compact, so that the compact-open topology on $Homeo_{\partial S}(S)$ coincides with the uniform convergence topology. We will work with some fixed hyperbolic or flat Riemannian metric $d$ on $S$ compatible with its topology with respect to which the boundary is geodesic. A system of symmetric neighbourhoods of the identity is thus given by $\{U_{\epsilon}\}_{\epsilon>0}$ where $U_{\epsilon}=V_{\epsilon}\cap V_{\epsilon}^{-1}$ and $V_{\epsilon}=\{g\in G\,|\,\forall x\in S\,d(x,gx)<\epsilon\}$.

  \begin{definition}
  \label{d: embedding}Let $\zeta(S)$ the minimum positive integer $m$ such that the point pushing map $\p:\pi_{1}(S_{m-1})\to \MCG(S_{m})$ is an embedding (a fortiori, of a torsion-free group), that is $m$ is equal to:
  \begin{itemize}
  \item $4$ if $S$ is a sphere
  \item $3$ if $S$ is a projective plane
  \item $2$ if $S$ is a torus or a disk 
  \item $1$ otherwise
  \end{itemize}
  
  \end{definition}

  
  \begin{theorem}
  \label{t: simple main}Let $S$ be a compact surface possibly with boundary and $m=\zeta(S)$. If $H\leq Homeo_{\partial S}(S)$ is $m$-transitive, then $H$ with the topology inherited from the compact-open topology is not Roelcke precompact.
  \end{theorem}
  
   \newcommand{\st}[0]{sufficiently transitive} 
  \renewcommand{\O}[0]{\mathcal{O}}
  For the remainder of the paper we will use the notation $G=Homeo_{\partial S}(S)$.
  Given $H\leq G$ we write $H_{p_{1},p_{2},\dots p_{r}}$ to denote the pointwise stabilizer of $\{p_{1},\dots p_{r}\}$.  
  
   In view of one of the open questions at the end it might be worthwhile to be somewhat more precise.
    \begin{definition}
  Let $H\leq\G$ and $\hat{p}$ a tuple of points in the interior of $S$ (possibly empty). Denote by $\O_{\hat{p}}(H)$ be the collection of orbits of points of $int(S_{\hat{p}})$ by $H_{\hat{p}}$.
  
  We say that $H$ is \emph{\st} if there exists a tuple of $\zeta(S)-1$ points $\hat{p}$ in $S$ and a non null-homotopic simple closed curve $\alpha$ in $S_{\hat{p}}$ such that $im(\alpha)$ is contained in some $O\in\O_{\hat{p}}(H)$.
  \end{definition}

  Notice that if $H$ is $\zeta(S)$-transitive then it is sufficiently transitive, so the following clearly implies \ref{t: simple main}.
  
  \begin{restatable}{theorem}{main}
  \label{t: main}Let $S$ be a compact surface. If $H\leq G$ is \st, then $H$ with the topology inherited from the compact-open topology is not Roelcke precompact.
  \end{restatable}

  \label{s: 2}
  \newcommand{\R}[0]{\mathbb{R}}
  \newcommand{\tup}[0]{p_{1},p_{2},\dots p_{m}}

  \newcommand{\pts}[0]{p_{1},p_{2},\dots p_{r}}
  \newcommand{\pmc}[1]{\MCG(S_{p_{1},\dots p_{r},#1})}
  
  \renewcommand{\H}[2]{H_{\hat{p},#1\to #2}}
  \newcommand{\HH}[2]{H_{\pts,#1\to #2}}

      \paragraph{\bf{Point pushing maps}} Fix some tuple $\hat{p}$ of points in $int(S)$. For any path $\alpha:J\to S_{\hat{p}}$ parametrized by an interval $J$ and any open $U\subseteq S_{\hat{p}}$ containing $im(\alpha)$ we choose a map $\p^{U}_{\alpha}\in\G$ supported in $U$ pushing $\alpha(0)$ to $\alpha(1)$ along $\alpha$. Very concretely: if $\alpha(0)\neq\alpha(1)$ and $\alpha$ is simple this just means a map supported in some open disk $U'\subseteq U$ which maps $\alpha(0)$ to $\alpha(1)$ and is supported in $U'$. If $\alpha=\alpha_{1}*\alpha_{2}\dots *\alpha_{q}$, where the $\alpha_{i}$ are simple paths, then
        we can take as $\p_{\alpha}^{U}=\p^{U}_{\alpha_{q}}\circ\p^{U}_{\alpha_{q-1}}\dots\circ\p^{U}_{\alpha_{1}}$.
        When $\nb{im(\alpha)}{\epsilon}\subseteq S_{\hat{p}}$ we use the expression $\p^{\epsilon}_{\alpha}$ to denote $\p^{\nb{im(\alpha)}{\epsilon}}_{\alpha}$, where $\nb{X}{\epsilon}=\{x\in S\,|\,d(x,X)<\epsilon\}$.
        
        We recall the following fundamental fact:
        \begin{fact}[Alexander lemma]
        \label{f: alexander} The mapping class group of a once-punctured disk is trivial.
        \end{fact}
        From it one can deduce the following (see Chapter $4$ of \cite{farb2011primer}):
        \begin{fact}
        \label{f: pushes}Let $\alpha,\beta$ be paths in $S_{\hat{p}}$ from a point $q$ to a point $q'$ and $U,V\subseteq S_{\hat{p}}$ open sets containing $im(\alpha)$ and $im(\beta)$ respectively.
        
        If $\alpha$ and $\beta$ are homotopic in $S_{\hat{p}}$ then
        $(\p_{\beta}^{V})^{-1}\p_{\alpha}^{U}$ and
        $\p_{\beta}^{V}(\p_{\alpha}^{U})^{-1}$
        represent the identity in
        $\MCG(S_{\hat{p},q})$ and $\MCG(S_{\hat{p},q'})$ respectively regardless of the choices made. If $\alpha,\beta$ are paths in $S_{\hat{p}}$  with  $\alpha(1)=\beta(0)$ and $U,V$ neighbourhoods of $im(\alpha)$ and $im(\beta)$ respectively, then
        $\p^{V}_{\beta}\circ\p^{U}_{\alpha}=\p^{U\cup V}_{\alpha*\beta}\circ h$ for some $h\in G_{\hat{p},\alpha(0)}$ trivial in $\MCG(S_{\hat{p},\alpha(0)})$.
        
        The induced map $\mathcal{P}:\pi_{1}(S_{\hat{p}},q)\to \MCG(S_{\hat{p},q})$ is a group homomorphism and an  a non-trivial group embedding of a torsion free group under the conditions described in definition \ref{d: embedding}.
        \end{fact}


        We will occasionally omit the superscript from $\mathcal{P}_{\alpha}$ in contexts where in light of the above the choice of neighbourhood and representative do not affect the resulting expression. 

        Given a curve $\alpha$, for any $0\leq s\leq t\leq 1$ let $\alpha_{s,t}$ be the restriction of $\alpha$ to $[s,t]$ reparametrized to the domain is $[0,1]$. If $s>t$ we let $\alpha_{s,t}=\alpha_{t,s}^{-1}$. 
         We extend the parametrization of a curve to $\R$ cyclically when needed. 
        
        We write $H_{p_{1},p_{2},\dots p_{r},p\to q}$
        to denote the collection of elements in $H_{p_{1},\dots p_{r}}$ sending $p$ to $q$.
        Since for any $g\in\G_{\hat{p},p\to q}$ the inner automorphism $f\mapsto f^{g}$ induces an isomorphism between $\MCG(S_{\hat{p},q})$ and $\MCG(S_{\hat{p},p})$ for which we will use the same expression $(-)^{g}$.

        \renewcommand{\pts}[0]{p_{1},p_{2},\dots p_{m-1}}
        \renewcommand{\pmc}[1]{\MCG(S_{p_{1},\dots p_{m},#1})}
        \newcommand{\hh}[2]{\H{\alpha(#1)}{\alpha(#2)}}
        
        \begin{lemma}
        	\label{l: pushing around} Let $\hat{p}=(p_{1},\dots p_{m-1})$ be an $(m-1)$-tuple of points in $int(S)$ and $\alpha$ an essential simple closed curve in $S_{\hat{p}}$ with $im(\alpha)\subseteq O\in\O_{\hat{p}}(H)$. Let $p_{m}=\alpha(0)$. Assume that to any point $q\in im(\alpha)$ we have assigned some subgroup $F_{q}\leq \MCG(S_{\hat{p},q})$ such that the following holds: 
        	\begin{enumerate}[(a)]
        		\item For any $q,q'\in Im(\alpha)$ and any $h\in H_{\hat{p},q\to q'}$ we have 
        		$F_{q'}^{h}=F_{q}$. \label{pa}
        		\item For any $t\in [0,1]$ there exists some $\delta_{t}>0$ such that for all $s\in (t-\delta_{t},t+\delta_{t})$ and sufficiently small $\epsilon$ the set $(\p_{\alpha_{s,t}}^{\epsilon})^{-1}H_{\hat{p},\alpha(t)\to\alpha(s)}$ maps to a subset of the set $F_{\alpha(t)}$ 
        	\end{enumerate}
        	Then there exists some $f\in F_{p_{m}}$ and some $h\in H_{p_{1},\dots p_{m}}$ such that $[h]=\p([\alpha])f\in \MCG(S_{p_{1},\dots p_{m}})$  .
        \end{lemma}
        \begin{proof}
          By compactness there is a finite tuple $0\leq t_{0}<t_{1}<\dots <t_{r}\leq 1$ such that $[0,1]\subseteq\bigcup_{i=0}^{r}(t_{i}-\delta_{i},t_{i}+\delta_{i})$. We may assume that this set is minimal for the property so that one can choose $q_{i}\in(t_{i-1},t_{i})$ for $1\leq i\leq r$ so that $q_{i}\in(t_{i-1},t_{i-1}+\delta_{i})\cap(t_{i}-\delta_{i},t_{i})$. Let also $q_{0}=0$ and $q_{r+1}=1$.
         \renewcommand{\p}[0]{\mathcal{P}^{\epsilon}}
        
         Pick $h_{0}\in H_{\hat{p},\alpha(t_{0})\to \alpha(0)}$,
        for each $0\leq i\leq r$ some $g_{i}\in\hh{t_{i}}{q_{i}}$ and for  $1\leq i\leq r+1$ some $h_{i}\in\hh{t_{i-1}}{q_{i}}$.
        Consider the group element $h=h_{r+1}g^{-1}_{r}\dots g_{1}^{-1} h_{1}g_{0}^{-1}$. One can check that $h\in H_{p_{1},p_{2},\dots p_{m}}$.
        
        By assumption we can write $g_{i}=\p_{\alpha_{t_{i},q_{i}}}c_{i}$ and $h_{i}=\p_{\alpha_{t_{i},q_{i+1}}}d_{i}$, where  $c_{i},d_{i}\in G_{\hat{p},\alpha(t_{i})}$ and $[c_{i}],[d_{i}]\in F_{\alpha(t_{i})}$ and any fixed $\epsilon<d(im(\alpha),\hat{p})$. 
        Alternatively, by property (\ref{pa}) we can write $g_{i}^{-1}=(\p_{\alpha_{t_{i},q_{i}}})^{-1}e_{i}$, where $e_{i}\in F_{q}$. 
        
         \renewcommand{\p}[0]{\mathcal{P}}
        Now we can write
         \begin{align*}
	         [h]=[h_{r+1}g^{-1}_{r}\dots g_{1}^{-1} \p_{\alpha_{t_{0},q_{1}}}d_{1} \underbracket{(\p_{\alpha_{t_{0},0}})^{-1}e_{0}}_{=g_{0}^{-1}\in H_{\hat{p},\alpha(0)\to\alpha(t_{0})}}]= \\
	         =[ h_{r+1}g^{-1}_{r}\dots  (\p_{\alpha_{t_{1} ,q_{1}}})^{-1}e_{1} \p_{\alpha_{t_{0},q_{1}}} (\p_{\alpha_{t_{0},0}})^{-1}e_{0}]\underbracket{[d_{1}^{g_{0}^{-1}}]}_{\in F_{\alpha(0)}}  \\
           =[h_{r+1}g^{-1}_{r}\dots  (\p_{\alpha_{t_{1} ,q_{1}}})^{-1} \p_{\alpha_{t_{0},q_{1}}}  (\p_{\alpha_{t_{0},0}})^{-1}[e_{0}][d_{1}^{g_{0}^{-1}}][e_{1}^{h_{1}g_{0}^{-1}} ]   \\
	         \dots = [\p_{\alpha_{t_{n},1}}(\p_{\alpha_{t_{n},q_{n}}})^{-1}\dots\p_{\alpha_{t_{0},q_{1}}}(\p_{\alpha_{t_{0},0}})^{-1}][h']
         \end{align*}   
              
         where $[h']$ is a product of elements in $F_{\alpha(0)}=F_{p_{m}}$ and thus itself in $F_{p_{m}}$. Notice that it is the same kind of manipulation that allows us to ignore the choice of particular point pushing maps in the expression above. In view of  \ref{f: pushes} we have $[h]\in\mathcal{P}([\alpha])F_{p_{m}}$.   
        \end{proof}
       
        \newcommand{\inn}[0]{\in\N}
        \renewcommand{\p}[0]{\mathcal{P}}
        \newcommand{\pp}[2]{\p_{\alpha_{{#1,#2}}}}
        
        \begin{lemma}
        \label{l: nesting}
        Let $\hat{p}=(p_{1},\dots p_{m-1})$, $\bar{p}=(\hat{p},p_{m})$, $H$ and $\alpha$ be as in Lemma \ref{l: pushing around}
        and assume that $m=\eta(S)$. Then given any $\delta>0$ there is a sequence $((f_{n},b_{n},w_{n}))_{n\in\N}$ of triples of elements of $G$ such that:
        \begin{enumerate}
        \item $f_{n}\in H_{\hat{p}}$ and  $w_{n}\in U_{\delta}\cap H_{\hat{p}}$
        \item $b_{n}=w_{n}^{-1}f_{n}\in G_{\bar{p}}$
        \item $\{[b_{n}]\}_{\inn}\subseteq \MCG(S_{\bar{p}})$ is an infinite set
        \end{enumerate}
        \end{lemma}
        \begin{proof}

        Notice that any element supported in $B(p_{m},\frac{1}{n})$ belongs to $U_{\frac{2}{n}}$.
        
        For any $q\in im(\alpha)$ and any $n>0$ let 
        $$
         A_{q,n}=\{[wh]\,|\,h\in H_{\hat{p}},w\in U_{\frac{1}{n}}\cap G_{\hat{p}}, wh(q)=q\}\subseteq \MCG(S_{\hat{p}}) 
        $$
        We distinguish two mutually exclusive cases: 
        \begin{enumerate}
        	\item $A_{q,n}$ is infinite for all $n>0$ 
        	\item  $A_{q,n}$ is finite for some $n>0$    
        \end{enumerate}
        Continuity of inner automorphisms and transitivity of $H_{\hat{p}}$ on $im(\alpha)$ implies that the same case holds for any $q\in im(\alpha)$. The first clearly implies the condition of the statement is satisfied, so from now on assume the second is always the case. 
        
        Define $F_{q}=\bigcap_{n>0}A_{q,n}$. This is the intersection of a descending chain of non-empty sets  which are eventually finite. Hence, it is a finite non-empty subset of $\MCG(S_{\hat{p},q})$. Continuity of inner automorphisms implies that for any $q,q'\in im(\alpha)$ and any $h\in H_{\hat{p},q\to q'}$ conjugation by $h$ induces an isomorphisms between $F_{q'}$ and $F_{q}$.
        
        \begin{lemma}
         	\label{l: subgroup}For any $q\in im(\alpha)$ the set $F_{q}$ is closed under multiplication. It is therefore a finite subgroup of $\MCG(S_{\hat{p},q})$. 
        \end{lemma}
        \begin{subproof}
        	Consider any $\eta_{1},\eta_{2}\in F_{q}$ and $n>0$. Then pick some $h_{1}\in H_{\hat{p}}$ that can be written as
        	$h_{1}=u_{1}a_{1}$, where $u_{1}\in G_{\frac{1}{2n}}$ and $a_{1}\in G_{\hat{p},q}$ maps to $\eta_{1}$.
        	By continuity of the group operations there exists some $n'>0$ such that $(U_{\frac{1}{n'}})^{a^{-1}}\subseteq U_{\frac{1}{2n}}$. Choose $h_{2}\in H_{\hat{p}}$ of the form   
        	$u_{2}a_{2}$, where $u_{2}\in G_{\frac{1}{n'}}$ and $a_{2}\in G_{\hat{p},q}$ maps to $\eta_{2}$. Then 
        	$h_{1}h_{2}=u_{1}u_{2}^{a_{1}^{-1}}a_{1}a_{2}\in U_{\frac{1}{n}}a_{1}a_{2}$ and $\eta_{1}\eta_{2}\in A_{q,n}$. Since $n$ is arbitrary, we are done.       
        \end{subproof}
        
        We are now in the position to apply Lemma \ref{l: pushing around} to the family $\{F_{q}\}_{q\in im(\alpha)}$. 
         This yields some element $h\in H_{p_{1},\dots p_{m}}$ (so that $[h]\in F_{p_{m}}$) whose image in $\MCG(S_{p_{1},\dots p_{m}})$ is of the form $f\p(\alpha)$, where $f\in F_{p_{m}}$. Since $F_{p_{m}}$ is a group we get that  $\p(\alpha)\in F_{p_{m}}$, contradicting the finiteness of $F_{p_{m}}$ and the fact that by the choice of $m$ the map 
        $\p:\pi_{1}(S_{\hat{p}},p_{m})\to \MCG(S_{\hat{p},p_{m}})$ is an embedding of a torsion-free group.
        \end{proof}



        \section{Discriminating between mapping classes}
        
        Given $\alpha,\alpha'$ arcs between the same pair of boundary points of $\partial S$ or simple closed curves in $S$ and given some subset $F\subseteq S$  we write $\alpha\simeq_{F}\alpha'$ to indicate that
        $F$ does not intersect $im(\alpha)\cup im(\alpha')$ and $\alpha$ is homotopic to $\alpha'$ in $S\setminus F$ (rel $\partial I$ in the arc case).

        \begin{definition}
        \label{d: delta}Let $\delta_{e}$ be small enough that for any $x,y$ at distance less than $ \delta_{e}$ the point $y$ is in the domain of injectivity of the exponential map at $x$, so that there is a unique geodesic segment from $x$ to $y$ of minimal length in $B(x,\delta_{e})\cap B(y,\delta_{e})$, which we denote by $[x,y]$.
        \end{definition}
        
        As it is well known, given two maps $f,f'$ from a space $X$ into $S$ such that $d(f(x),f'(x))<\delta_{e}$ for any $x\in X$ a uniform parametrizataion of the geodesic $[f(x),f'(x)]$ yields a homotopy in $S$ between the two maps $f$ and $f'$.
        
        An immediate consequence of this is the following observation:
        \begin{observation}
        \label{o: homotopic to the identity}Let $\bar{p}=(p_{1},p_{2}\dots p_{m})$ a tuple of points in $int(S)$ and let $u\in U_{\delta_{e}}\cap G_{\bar{p}}$. Then $u$ represents the identity in $\MCG(S_{\bar{p}})$.
        \end{observation}
        
        Another one is the Lemma below, which for future reference we state in a degree of generality not needed here. Its content is rather standard but we include a proof for the sake of completeness.

        \begin{definition}
        We say that a path connected subset $U\subseteq X$ is $\epsilon$-branching
        if there exist $x_{1},x_{2},x_{3}\in U$ and paths $\gamma_{ij}$ between $x_{i}$ and $x_{j}$ for distinct $i,j$ such that $\{i,j,k\}=\{1,2,3\}$ the path $\gamma_{ij}$ does not contain points at distance $<\epsilon$ from $x_{k}$. Clearly, any path connected open set is $\epsilon$-branching for some $\epsilon$.
        \end{definition}

        \begin{lemma}
        \label{l: non-linearity} Let $D_{1},D_{2},\dots D_{m}\subseteq S$ be $\epsilon$-branching disjoint closed subdisks of $S$. Then there is some constant $\delta $ depending only on the metric such that the following holds. If $\alpha,\beta:I\to S$ are two loops in $S\setminus\bigcup^{m}_{i=1}D_{i}$ with common endpoints and $d(\alpha(t),\beta(t))<\min\{\delta_{e},\epsilon\}$ for all $t\in I$, then $\alpha_{\simeq_{\bigcup_{i=1}^{m}}D_{i}}\beta$.
        \end{lemma}
        \begin{proof}

        
        For $1\leq i\leq m$ let the non-linearity of $D_{i}$ be witnessed by a triple $(x_{j}^{i})_{1\leq j\leq 3}$ of points in $D_{i}$ and paths $(\gamma_{jk}^{i})_{1\leq j<k\leq 3}$ in $D_{i}$ between the points in the triple.
        
        Let $H:I^{2}\to S$ be the homotopy between $\alpha$ and $\beta$ where
        $s\mapsto H(t,s)$ is a constant parametrization of the geodesic segment $[\alpha(t),\beta(t)]$.
        
        The following claim is easy to verify using rudiments of differential topology.
        \begin{claim}
        \label{cl: auxiliary}
        After performing a small perturbation of $\alpha$, $\beta$, $H$ and $\gamma^{i}_{jk}$
        we may assume that $\alpha,\beta$ and $H$ are differentiable, that $H$ is transverse to $\partial D_{i}$, $\gamma^{i}_{jk}$ and the points $x^{i}_{j}$, while preserving the property that $d(H(t,s_{1}),H(t,s_{2}))<\delta$ for any $t,s_{1},s_{2}\in I$.
        After modifying the $\gamma^{i}_{jk}$ further and shrinking $D_{i}$ we may additionally assume that $\partial D_{i}=\bigcup_{1\leq j<k\leq 3}\gamma^{i}_{jk}$. In particular $\gamma^{i}_{jk}$ meet only at the endpoints and they do so smoothly.
        
        \end{claim}

        \begin{lemma}
        For any connected component $C$ of $H^{-1}(D_{i})$ the map $H_{\restriction}:\partial C\to\partial D_{i}$ is nullhomotopic. \label{l: finiteness}
        \end{lemma}
        \begin{subproof}
        For convenience we will drop the superindex $i$ for the remainder of this proof.
        For $1\leq j,k\leq 3, j\neq k$ let $Y_{j}=H^{-1}(\{x_{j}\})$ and $Y_{jk}=H^{-1}(im(\gamma_{jk}))$.
        Let also $Y=\bigcup_{j=1}^{3}Y_{j}$.
        By transversality, we have the following:
        \begin{itemize}
        	\item $Y_{j}$ is finite for $1\leq j\leq 3$
        	\item $Y_{jk}$ is a union of finitely many arcs between points in $Y_{j}\cup Y_{k}$  (potentially arcs between two points in $Y_{j}$ or in $Y_{k}$)
        	\item $H^{-1}(\partial C)=\bigcup_{j\neq k}Y_{jk}$ is the image of a simple closed curve
        	\item $y\in Y_{j}$ is the endpoint of exactly one arc in $H^{-1}(\gamma_{kj})$ for $k\neq j$ (say $\gamma_{jk}=\gamma_{kj}$ ).
        \end{itemize}
        
        This makes  $\partial C$ into a simplicial complex with  $Y=\coprod_{j=1}^{3}Y_{j}$ and the components of $Y_{jk}$ as edges. Let $\pi$ be the projection of $I\times I$ onto the second coordinate.
        
        Notice that if $\{i,j,k\}=\{1,2,3\}$ and $x,y,z\in U$, then we cannot have $\pi(x)\in[\pi(y),\pi(z)]$ for
        $x\in Y_{i}$ and $[\pi(y),\pi(z)]\subseteq Y_{jk}$, since this would imply the existence of $w\in im(\gamma_{jk})$ such that $d(w,x_{i})<\delta$.
        
        We show that $H_{\restriction}:\partial C\to\partial D$ is null-homotopic by induction on the number of vertices in $Y$.
        
        Number the vertices of $Y$ cyclically as $y_{1},y_{2},\dots y_{N}$.
        Let $\tau_{i}=\pi(y_{i})$ and let $c_{i}\in\{1,2,3\}$ be such that $y_{i}\in Y_{c_{i}}$ and $d_{i}\in\{12,23,31\}$ such that $[y_{i},y_{i+1}]\subseteq Y_{d_{i}}$.

        Let $i$ be such that $\tau_{i}$ is minimal. Without loss of generality we may assume $i=2$ and $c_{2}=1$. The case $\{c_{1},c_{3}\}=\{2,3\}$ is excluded by our condition on $f$.
        
        Consider first the case in which $\{c_{1},c_{3}\}=\{1,k\}$, with $k\neq 1$.
        Without loss of generality we can assume that $k=2=c_{1}$. We cannot have
        $\tau_{1}\leq\tau_{3}$ since then $\tau_{1}\in [\tau_{2},\tau_{3}]$. But $c_{1}=2$ and $d_{1}=12$, which forces $d_{2}=31$, resulting in a collision. So $\tau_{1}>\tau_{3}$. We may assume $c_{1}=1$. Necessarily $d_{1}=13$
        which in turn implies $d_{0}=12$.
        
        At this point we can remove vertices $y_{1}$ and $y_{2}$ connect $y_{0}$ and $y_{3}$ directly with an edge of type $12$ (so compatible with the other incident edges at the endpoints). The non-collision condition continues to be satisfied, since the image of the new edge is contained in the image of two edges of the same type in the old configuration: $[\tau_{0},\tau_{3}]\subseteq[\tau_{0},\tau_{1}]\cup[\tau_{2},\tau_{3}]$, as $\tau_{1}\in(\tau_{2},\tau_{3})$.
        
        We have replaced a path of type $2_{12}1_{13}1_{12}2$ (resp $1_{12}1_{13}1_{12}2$) with one of type $2_{12}2$ (resp. $1_{12}2$), both of which represent homotopic paths between $x_{2}$ and itself (resp. $x_{1}$ and $x_{2}$) on the circle $\partial D$. Since the new configuration corresponds to a null-homotopic path by the induction hypothesis, so does the original.
        
        The only case left is that in which $\{c_{1},c_{3}\}=\{1\}$. We may assume $d_{1}=13$ and thus $d_{2}=12$, as well as $\tau_{1}<\tau_{3}$ and we can apply an argument identical to the one of the previous case.
        
        %
        \end{subproof}
        Given $H$ satisfying the conclusion of the sublemma, we can redefine $H$ on each component $C$ as above by extending $H_{\partial C}$ to a map from $C$ to $\partial D_{i}$. Postcomposing with a homeomorphism that slightly enlarges each $D_{i}$ while fixing $\alpha$ and $\beta$ yields a homotopy witnessing  $\alpha\simeq_{\bigcup_{i=1}^{m} D_{i}}\beta$.
        
        \end{proof}

        In the proof of the theorem below we will use the following (see \cite{farb2011primer}, Proposition 2.8)
  \begin{fact}
	  \label{f: action of mapping class group}The action of the homeomorphism group descends to a well-defined action of the mapping class group on free homotopy classes of curves in the surface and homotopy classes relative to the boundary of arcs between points in the boundary. The kernel of the action of the mapping class group on the union of the two families is trivial.	  
  \end{fact}

        \main*
        
        \begin{proof}
        
        \newcommand{\e}[0]{4\delta_{0}}
        \newcommand{\ee}[0]{2\delta_{0}}
        \newcommand{\eee}[0]{\delta_{0}}
        
        Let $m=\zeta(S)$. Let $\hat{p}=(p_{1},\dots p_{m-1})$ and $\bar{p}=(\hat{p},p_{m})$ be tuples of distinct points as in Lemma \ref{l: nesting}.
        $$
        \delta_{0}= \frac{1}{7}\min\{\delta_{e},\min\{d(p_{i},\partial S),d(p_{i},p_{j})\,|\,1\leq i<j<m\}\}.
        $$
        where $\delta_{e}$ is as in definition \ref{d: embedding}
        
        Roelcke precompactness of $H$ implies the existence of a finite collection $h_{1},\dots h_{M}\in H$ such that
        $H\subseteq\bigcup_{j=1}^{M}U_{\eee}h_{j}U_{\eee}$.
        

        Let $f_{n}\in H_{p_{1},p_{2},\dots p_{m-1}}$, $b_{n}\in G_{\bar{p}}$, $w_{n}\in U_{\delta_{1}}\cap G_{\hat{p}}$ for $n\in\N$ be as given by Lemma \ref{l: nesting} applied to $H$ with constant $\delta=\delta_{0}$. In view of the above, for any $n$ there exists $j_{n}\in\{1,2,\dots M\}$ and $u_{n},v_{n}\in U_{\eee}$ such that $u_{n}f_{n}=h_{j_{n}}v_{n}$.
        Notice that for any $1\leq i\leq m$ we have $v_{n}(p_{i})\in B(p_{i},\eee)\subseteq v_{n}(B(p_{i},\ee))$ (for the last inclusion, notice that $v_{n}^{-1}\in U_{\delta_{0}}$).
        If follows that
        $$
        u_{n}f_{n}p_{i}=h_{j_{n}}v_{n}p_{i}\in h_{j_{_{n}}}(B(p_{i},\eee))\subseteq h_{j_{n}}v_{j_{n}}B(p_{i},\ee)=u_{n}f_{n}B(p_{i},\ee)
        $$
        
        
        Up to passing to a subsequence again we may assume that $j_{n}$ equals some constant $j_{0}$ for all $n$.
        For $1\leq i\leq m$ let $K_{i}=h_{j_{0}}\bar{B}(p_{i},\eee)$.
        
        \newcommand{\rin}[0]{r_{n}}    
        \newcommand{\rinn}[0]{r_{n'}}
        \newcommand{\qin}[0]{q^{i}_{n}}
        \newcommand{\qinn}[0]{q^{i}_{n'}}

        For $n>0$ write $q_{n}^{i}=u_{n}f_{n}p_{i}$. We know that $d(f_{n}p_{i},p_{i})<\delta\leq\eee$ and that $d(f_{n}p_{i},q_{n}^{i})<\eee$ ($w_{n}\in U_{\delta}$) so that $d(q_{n}^{i},p_{m}^{i})<\ee$.
        We choose a series of paths between these points as follows. For $n\inn$, $1\leq i\leq m$ take any path $\nu_{n}^{i}$ in $B(p_{m},\ee)$ from $p_{i}$ to $\qin$.

        \begin{lemma}
        \label{l: finitely many rhos}After passing to a subsequence we can choose for $1\leq i\leq m$ and $n,n'\geq 0$ paths $\mu_{n',n}^{i}$ from $q_{n}^{i}$ to $q_{n'}^{i}$ and $\epsilon\in (0,\delta_{0})$  with the following properties.
        \begin{itemize}
         	\item $N_{\epsilon}(im(\mu_{n',n}^{i}))\subseteq h_{j_{0}}(B(p_{i},\eee))\setminus\{q_{n}^{i}\,|\,n\inn,1\leq l\leq m, l\neq i\}$
        	\item Neither of the points $q_{n'}^{j}$ or $q_{n}^{j}$ belongs to $W_{n',n}^{i}$ for $1\leq j\leq m$, $j\neq m$.
        	
        	\item Let
        	$$P_{n',n}^{i}=\prod_{1\leq l<i}\p^{\delta_{0}}_{\nu^{l}_{n'}}\prod_{i<l\leq m}\p^{\delta_{0}}_{\nu^{l}_{n}}\in G$$
        	and notice that it fixes the ball $B(p_{i},\delta_{0})$. Consider also the loops at $p_{i}$:
        	\begin{align*}
        		\sigma_{n',n}^{i}:=&\nu_{n'}^{i}*\mu_{n',n}^{i}*(\nu_{n}^{i})^{-1} \\
        		\rho_{n',n}^{i}:=&(P_{n',n}^{i})^{-1}\circ\sigma_{n',n}^{i}=\nu_{n'}^{i}*((P^{i}_{n',n})^{-1}\circ\mu_{n',n}^{i})*(\nu_{n}^{i})^{-1}
        	\end{align*}
        	Then for any $1\leq i\leq m$ the image $\mathcal{R}^{i}$ of the set $\{\rho_{n',n}^{i}\}_{n,n'\in\N}$ in $\pi_{1}(S_{\bar{p}\setminus\{p_{i}\}},p_{i})$ is finite.
        \end{itemize}

        \end{lemma}
        
        \begin{subproof}
          After passing to a subsequence we may assume that the points $q_{n}^{i}$ converge to a point $q_{\infty}^{i}$ for all $1\leq i\leq m$. If we let  $\mathcal{Q}^{i}=\{q_{n}^{i},q_{\infty}^{i}\,|\,n\in\N\}$,
        then each of the sets
        $h_{j_{0}}(B(p_{i}))\setminus\bigcup_{l\neq i}\mathcal{Q}^{i}$ are connected.
        
        We start with the following observation:
        \begin{observation}
        	\label{o: finitely many}	Let $W$ be the interior of an embedded disk, $x\in S$, $U,U'$ two open sets in $S$ containing $x$ with $\bar{U'}\subseteq U$ and $\mathcal{C}$ some countable closed sets of points. Then there exits some finite family $\mathcal{L}$ of arcs between points $W\setminus\mathcal{C}$ such that for any two points $x,y\in (W\cap U')\setminus\mathcal{C}$ there exists some $\lambda_{x,y}\in\mathcal{L}$ and arcs
        	$\beta_{x,y}$ and $\gamma_{x,y}$ in $(W\cap U)\setminus\mathcal{C}$ such that
        	$$\beta_{x,y}(0)=x,\,\,\beta_{x,y}(1)=\lambda_{x,y}(0),\,\, \lambda_{x,y}(1)=\gamma_{x,y}(0),\,\,\gamma_{x,y}(1)=y.$$
        \end{observation}
        \begin{subsubproof}
        	This follows from a standard argument.  By compactness, there are finitely many connected components $U_{1},\dots U_{q}$ of $U\cap \bar{W}$ such that $\bar{W}\cap\bar{U}'\subseteq\bigcup_{i=1}^{q}U_{i}$. Choose some point $y_{i}\in U_{i}\cap (W\setminus\mathcal{C})$ for $1\leq i\leq q$. Now  $W\setminus\mathcal{C}$ is connected and as $\mathcal{L}$ it suffices to choose some finite collection of arcs $\lambda_{i,j}$ in $W\setminus\mathcal{C}$ between $y_{i}$ and $y_{j}$ for every pair of distinct $i,j\in\{1,2,\dots q\}$.
        \end{subsubproof}
        
        For fixed $1\leq i\leq m$ let $\lambda_{-},\beta_{-},\gamma_{-}$ be the families that result from applying Observation \ref{o: finitely many} with $W=h_{j_{0}}(B(p_{i},\delta_{0}))\setminus\bigcup_{l\neq i}\mathcal{Q}^{l}$, $x=p_{i}$, $U=B(p_{i},3\delta_{0})$ and $U'=B(p_{i},2\delta_{0})$ and $\mathcal{C}=\mathcal{Q}$.
        
        For simplicity, write  $\lambda^{i}_{n,n'}=\lambda_{q_{n}^{i},q_{n'}^{i}}$ and so forth. Let $\mu_{n,n'}^{i}=\beta_{n,n'}^{i}*\lambda_{n,n'}^{i}*\gamma_{n,n'}^{i}$.
        We can assume that $\lambda^{i}_{n,n'}$ and hence $\mu_{n,n'}^{i}$ and $\sigma_{n,n'}^{i}$  intersects $\bar{B}(p_{l},2\delta_{0})$ in a finite collection of arcs for $1\leq l\leq m$, $l\neq m$.
        
        We first note that it follows from the finiteness of $\mathcal{L}$ that the image $\mathcal{S}^{i}$ of the collection $\{\sigma_{n,n'}^{i}\}$ in $\pi_{1}(S_{\bar{p}\setminus\{p_{i}\}},p_{i})$ in the statement is finite.
        
        Indeed, if $\lambda_{n_{0},n_{1}}^{i}=\lambda_{n_{2},n_{3}}^{i}=:\lambda$ then we can rewrite $(\sigma_{n_{0},n_{1}}^{i})^{-1}\sigma_{n_{2},n_{3}}^{i}$ as
        $$
        \nu_{n_{1}}^{i}*(\gamma_{n_{0},n_{1}}^{i})^{-1}*\lambda^{-1}*(\beta_{n_{0},n_{1}}^{i})^{-1}*(\nu_{n_{0}}^{i})^{-1}*\nu_{n_{2}}^{i}*\beta_{n_{0},n_{1}}^{i}*\lambda*\gamma_{n_{2},n_{3}}^{i}*(\nu_{n_{3}}^{i})^{-1}
        $$
        Since the simple closed paths $(\beta_{n_{0},n_{1}}^{i})^{-1}*(\nu_{n_{0}}^{i})^{-1}*\nu_{n_{2}}^{i}*\beta_{n_{0},n_{1}}^{i}$ and
        $\nu_{n_{1}}^{i}*(\gamma_{n_{0},n_{1}}^{i})^{-1}*(\gamma_{n_{2},n_{3}}^{i})*(\nu_{n_{3}}^{i})^{-1}$ are contained in
        a disk in $S_{\bar{p}\setminus\{p_{i}\}}$ they must be homotopic to the identity there. Therefore so is $(\sigma_{n_{0},n_{1}}^{i})^{-1}\sigma_{n_{2},n_{3}}^{i}$.

        In order to ensure the finiteness of the homotopy classes of the $\rho^{i}_{n',n}$ we will need the following two standard observations, whose proof we include for the sake of completeness:
        \begin{observation}
        	\label{o: perturbing a curve}Let $T$ be a punctured surface, $U\subseteq T$ an open disk
        	, $p\in U$ and $\alpha$ a curve based at some $*\nin U$. Then given any $q\nin im(\alpha)$ and any $g\in G$ supported in $U$  sending $q$ to $p$, the homotopy class of $g(\alpha)$ in $T_{p}$ depends only on the connected component of $U\setminus im(\alpha)$ to which the point $q$ belongs.
        \end{observation}
        \begin{subsubproof}
        	Assume that $q$ and $q'$ are two points in the same connected component $C$ of $U\setminus im(\alpha)$ and $g,g'$ supported in $U$ map $q$ and $q'$ respectively to $p$. If
        	$g(\alpha)\nsimeq_{p} g'(\alpha)$, then $\alpha\nsimeq_{q} g^{-1}g'(\alpha)$.
        	Take some $g''$ supported in $C$ mapping $q$ to $q'$, then
        	$\alpha=g''\alpha\nsimeq_{q'}g''g^{-1}g\alpha$. But by Fact \ref{f: alexander} the map $g''g^{-1}g$ must be the identity in $\MCG(T_{q'})$: a contradiction.
        \end{subsubproof}

        The following follows from a standard application of transversality.
        \begin{observation}
        	For any any embedded disk $D$ contained in some open set $W$ and any finite collection $\mathcal{M}$ of disjoint simple arcs between points outside of $W$ there exists another closed disk $D'$ with $D\subseteq D'\subseteq W$ such that each  $\mu\in\mathcal{M}$ intersects $D'$ in a finite collection of arcs.
        \end{observation}

        Consider the loop $\rho_{n',n}^{i}:=(P_{n',n}^{i})^{-1}\circ\sigma_{n',n}^{i}=\nu_{n'}^{i}*((P^{i}_{n',n})^{-1}\circ\mu_{n',n}^{i})*(\nu_{n}^{i})^{-1}$. This is a loop in $S_{p_{1},\dots p_{i-1},p_{i+1},\dots p_{m}}$, since im  $q_{n}^{l},q_{n'}^{l}\notin im(\alpha)$ for $l\neq i$.
        
        For $1\leq l\leq m$, $l\neq i$ let $F_{l}$ be a closed embedded disk with $B(p_{l},2\delta_{0})\subseteq F_{l}\subseteq B(p_{l},3\delta_{0})$ and such that $\lambda_{n,n'}^{i}$ (and hence $\sigma_{n,n'}^{i}$) intersects $F_{l}$ in a finite collection of arcs.
        
        By iteratively applying Observation \ref{o: perturbing a curve}, we conclude that the homotopy class of $\rho_{n',n}^{i}$ depends only on the connected component of $F_{l}\setminus im(\mu_{n',n}^{i})$ in which $q_{n'}^{l}$ lies for $1\leq l<i$
        as well as the connected component of $F_{l}\setminus im(\mu_{n',n}^{i})$ in which one $q_{n}^{l}$ lies for $i<l\leq m$. Since in each case there are only finitely many possibilities the result follows.
        
        The existence of $\epsilon$ such that $N_{\epsilon}(im(\mu^{i}_{n',n}))$ does not intersect $\mathcal{Q}^{l}$ for $1\leq l\leq m$, $l\neq i$ is clear.
        \end{subproof}
        
        Choose $\eta>0$ such that
        $h_{j_{0}}B(p_{i},\eee)$ is $\eta$-branching
        and let $\delta_{1}=\min\{\eta,\delta_{0}\}$.
        
        By another application of Roelcke precompactness, after passing to a subsequence again we may assume that for all distinct $n<n'\leq N$ there exists  $u'_{n,n'},u''_{n,n'}\in U_{\delta_{1}}$ such that $u''u_{n'}f_{n'}=hu_{n}f_{n}u'$. Fix $n'=0$ and for the time being also $n>0$ and let $u'=u'_{0,n},u''=u''_{0,n}$.
        
        Let also $\alpha$ be any essential simple closed curve or simple arc between boundary points $\alpha$ that does not intersect any of the disks $\bar{B}(p_{i},\e)$. It follows that $u'\circ\alpha$ does not intersect $B(p_{i},\ee)$ for any $1\leq i\leq m$. Therefore $u_{n}f_{n}u'\circ\alpha$ does not intersect
        $h_{j_{0}}B(p_{i},\eee)\subseteq u_{n}f_{n}B(p_{i},\ee)$ for any $1\leq i\leq m$ and the same holds for $u_{0}f_{0}\circ\alpha$.   
        
  \setlength{\abovedisplayskip}{7pt}
  \setlength{\belowdisplayskip}{7pt}
        On the other hand, since $u''\in U_{\delta}$ from the choice of $\delta_{1}$ and $\delta_{0}$ it follows that $$d((u_{0}f_{0}\circ\alpha)(t),(u_{n}f_{n}u'\circ\alpha)(t))<\min\{\eta,\delta_{e}\}.$$
        Lemma \ref{l: non-linearity} yields that $u_{0}f_{0}\circ\alpha\simeq_{K} u_{n}f_{n}u'\circ\alpha $, where $K=\bigcup_{i=1}^{m}K_{i}$.
   
        Now, $\alpha$ and $u'\circ\alpha$ can be homotoped to each other through a homotopy with values in
        $S\setminus \bigcup_{i=1}^{m}B(p_{i},\ee)$, it follows that $ u_{n}f_{n}u'\circ\alpha\simeq_{K}u_{n}f_{n}\circ\alpha$

        Consider the commuting products $\p_{\mu_{0,n}}=\prod_{i=1}^{m}\p_{\mu^{i}_{0,n}}^{\epsilon}$ and
        $\p_{\nu_{n}}=\prod_{i=1}^{m}\p_{\nu^{i}_{n}}^{\delta_{0}}$. We may assume $(\p^{\delta_{0}}_{\nu^{i}_{0,n}})^{-1}=\p^{\delta_{0}}_{(\nu^{i}_{0,n})^{-1}}$.
        Since $\p_{\mu_{0,n}}$ is supported in $K$, for any $\alpha$ as above we have $\p_{\mu_{0,n} }u_{0}f_{0}\circ\alpha\simeq_{K}u_{0}f_{0}\circ\alpha$. It follows that  $$\p_{\mu_{0,n}}u_{0}f_{0}\circ\alpha\simeq_{K}u_{n}f_{n}$$
        Since $\qin\in K$ and $\p_{\nu_{n}}^{-1}(\qin)=p_{i}$ we have
        $$\p_{\nu_{n}}^{-1}\p_{\mu_{0,n}}u_{0}f_{0}\circ\alpha\simeq_{\bar{p}}\p_{\nu_{n}^{-1}}u_{n}f_{n}\circ\alpha .$$        
        Since both $\p_{\nu_{n}}^{-1}\p_{\mu_{0,n}}u_{0}f_{0}$ and $\p_{\nu_{n}}^{-1}u_{n}f_{n}$ belong to $H_{\bar{p}}$ and $\alpha$ ranges over a set of curves/arcs containing representatives of all homotopy classes in $S_{\bar{p}}$ we conclude that the images of these two elements in $\MCG(S_{\bar{p}})$ are equal. Expanding $f_{n}$ as $w_{n}b_{n}$ with $b_{n}\in G_{\bar{p}}$ we obtain
        $$
        [\p_{\nu_{n}}^{-1}\p_{\mu_{0,n}}\p_{\nu_{0}}][\p_{\nu_{0}}^{-1}u_{0}w_{0}][b_{0}]=[\p_{\nu_{n}}^{-1}u_{n}w_{n}][b_{n}].
        $$        
        Notice that
        $\p_{\nu_{n}}^{-1}u_{n}w_{n}\in G_{\bar{p}}$.
        On the other hand, we have $\p_{\nu^{i}_{n}}^{\delta_{0}}\in U_{6\delta_{0}}$, since it is supported on $N_{\delta_{0}}(im(\nu_{n}))\subseteq B(p_{i},3\delta_{0})$. The latter also implies that the supports of  $\p_{\nu^{i}_{n}}^{\delta_{0}}$ for different values of $i$ are disjoint and hence $\p_{\nu_{n}}\in U_{6\delta_{0}}$. It follows that $\p_{\nu_{n}}^{-1}u_{n}w_{n}\in U_{6\delta_{0}+\delta_{0}+\delta_{1}}\subseteq U_{\delta_{e}}$. By Observation \ref{o: homotopic to the identity} we then have  $[\p_{\nu_{n}}^{-1}u_{n}w_{n}]=1\in \MCG(S_{\bar{p}})$ and like-wise
        $[\p_{\nu_{0}}^{-1}u_{0}w_{0}]=1$.
        
        We now claim that   $[\p_{\nu_{n}}^{-1}\p_{\mu_{0,n}}\p_{\nu_{0}}]\in \p(\mathcal{R}^{1})\cdot\p(\mathcal{R}^{2})\dots \p(\mathcal{R}^{m})$, where $\mathcal{R}^{i}$ is the finite set in Lemma \ref{l: finitely many rhos}. This concludes the proof, since letting $n>0$ range we obtain that the set $\{[b_{n}]\}_{n\in \N}$ has to be finite, contradicting the choice of $b_{n}$.
        
        The claim follows from a straightforward algebraic manipulation, which we sketch below, justifiably ignoring the choice of push representatives: 
        \begin{equation*}
        \begin{gathered}
        	[\p_{\nu_{n}}^{-1}\p_{\mu_{0,n}}\p_{\nu_{0}}]=
        	[(\p_{\nu_{n}^{1}}^{-1}\prod_{i=2}^{m}\p_{\nu_{n}^{i}}^{-1})
        	(\p_{\mu_{0,n}^{1}}\prod_{i=2}^{m} \p_{\mu_{0,n}^{i}})
        	(\p_{\nu_{0}^{1}}\prod_{i=2}^{m}\p_{\nu_{0}^{i}})]\\
        	= [ \p_{\nu_{n}^{1}}^{-1}\p_{\mu_{0,n}^{1}}^{\prod_{i=2}^{m}\p_{\nu_{n}^{i}}}\p_{\nu_{0}^{1}}
        	(\prod_{i=2}^{m}\p_{\nu_{n}^{i}}^{-1}
        	\prod_{i=2}^{m} \p_{\mu_{0,n}^{i}})^{\p_{\nu_{0}^{1}}}
        	\prod_{i=2}^{m}\p_{\nu_{0}^{i}}]\\=
        	\p([\rho^{1}_{0,n}])[
        	(\prod_{i=2}^{m}\p_{\nu_{n}^{i}}^{-1}
        	\prod_{i=2}^{m} \p_{\mu_{0,n}^{i}}
        	\prod_{i=2}^{m}\p_{\nu_{0}^{i}})^{\p_{\nu_{0}^{1}}}]\\
        	=\p([\rho^{1}_{0,n}])\p([\rho^{2}_{0,n}])\dots \p([\rho^{m}_{0,n}])
        \end{gathered}
        \end{equation*}

        \end{proof}

\section{Questions}
  
  The torus admits a $1$-transitive Roelcke precompact subgroup of homeomorphisms, namely the one given by the diagonal action of $Homeo(S^{1})\times Homeo(S^{1})$ on $S^{1}\times S^{1}$.
  
  \begin{question}
  Is the bound $\zeta(S)$ on the degree of transitivity on $S$ of a Roelcke precompact subgroup of $Homeo_{\partial S}(S)$ sharp in those cases in which $\zeta(S)>1$ and $S$ is not a torus?
  \end{question}
  \begin{question}
  	Is the subgroup of $Homeo(S^{1}\times S^{1})$ given above the only $1$-transitive Roelcke precompact subgroup on $S^{1}\times S^{1}$ up to conjugacy? 
  \end{question}
  
  \begin{question}
  Can a transitive (resp. $\omega$-transitive) group of homeomorphisms of an $m$-manifold, $m\geq 3$ be Roelcke precompact?
  \end{question}

  %
  

  \bibliographystyle{plain}
  \bibliography{bibliography}
  
\end{document}